\documentclass{amsart}

\usepackage{amsmath}
\usepackage{amsfonts}
\usepackage{amssymb}
\usepackage{hyperref}
\usepackage{subfig}
\usepackage{graphicx}
\usepackage{array}
\usepackage{indentfirst}
\usepackage{cite}
\usepackage{enumerate}

\newtheorem{theorem}{Theorem}[section]
\newtheorem{lemma}[theorem]{Lemma}
\newtheorem{corollary}[theorem]{Corollary}

\theoremstyle{definition}
\newtheorem{definition}[theorem]{Definition}

\theoremstyle{remark}
\newtheorem{remark}[theorem]{Remark}

\numberwithin{equation}{section}

\begin{document}

\title[Uniqueness of planar tangent maps]{Uniqueness of planar tangent maps \\ in the modified Ericksen model}

\author{Onur Alper}
\address{Department of Mathematics, Purdue University \\ 150 N. University Street, West Lafayette, IN 47907, USA}
\email{oalper@purdue.edu}

\begin{abstract}
We prove the uniqueness of homogeneous blow-up limits of maps minimizing the modified Ericksen energy for nematic liquid crystals in a planar domain.
The proof is based on the Weiss monotonicity formula, and a blow-up argument, originally due to Allard and Almgren \cite{AA} for minimal surfaces, 
and L. Simon \cite{SL} for energy-minimizing maps into analytic targets, which exploits the integrability of certain Jacobi fields.
\end{abstract}

\maketitle

\section{Introduction}

The uniqueness of homogeneous blow-up limits at zeros, critical or singular points of maps that minimize a particular energy
is an important question in geometric analysis. A map that has a unique homogeneous blow-up limit at a point admits a first-order approximation near it. 
Determining the local behavior near such a point often yields crucial information regarding the structure and stability of nodal, critical or singular set.
We prove the uniqueness of homogeneous blow-up limits at the point defects of minimizers of the modified Ericksen energy for nematic liquid crystals in a planar domain.

\subsection{The Ericksen Model}
We consider a probability distribution $\ell$ of unit vectors for the direction of a symmetric, elongated liquid crystal molecule at a given point in a spatial domain $\Omega$. Its first moment, $\left \langle l \right \rangle$, must vanish, due to the symmetry of individual molecules.
Hence, we consider its second moment, $\left \langle \ell \otimes \ell \right \rangle$, to capture the anisotropy of the liquid crystal sample in $\Omega$.
Under the uniaxial assumption that $\left \langle \ell \otimes \ell \right \rangle - \frac{1}{3} id$ has two equal eigenvalues, we have:
$$
\left \langle l \otimes l \right \rangle - \frac{1}{3} id = s \left [ (n \otimes n ) - \frac{1}{3} id \right ],
$$
where $| n | = 1$, $s = \frac{3}{2} \left \langle l \otimes n \right \rangle^2 - \frac{1}{2} \in \left [ -1/2, 1 \right ]$.
For constant $s$, the Oseen-Frank energy is defined as $\int_{\Omega} W(n) \, \mathrm{d}x$, where
\begin{equation*}
 W(n) = \kappa_1 | \mathrm{div} n |^2 + \kappa_2 | n \cdot \mathrm{curl} n |^2 + \kappa_3 | n \times \mathrm{curl} n |^2 
+ \left ( \kappa_2 + \kappa_4 \right ) \left [ \mathrm{tr} ( \nabla n )^2 - ( \mathrm{div}n )^2 \right ],
\end{equation*}
for material constants $\kappa_1$, $\kappa_2$, $\kappa_3 > 0$.

In \cite{HKL}, it was proved that the set of points in a neighborhood of which the director field $n$ fails to be continuous,
namely the {\emph{defect set}}, has Hausdorff dimension strictly less than $1$. In particular, the Oseen-Frank model does not admit
energy-minimizing configurations with line defects.
The original Ericksen model for nematic liquid crystals was proposed in \cite{Ericksen} 
as a generalization of the Oseen-Frank model with variable $s$, in order to allow enegy-minimizing configurations with line defects.
In \cite{Ericksen, Maddocks}, the Ericksen energy $\int_{\Omega} X(s,n)) \, \mathrm{d}x$ is introduced, where
$$
X(s,n) = s^2 W(n) + \kappa_5 | \nabla s|^2 + \kappa_6 | \nabla s \cdot n |^2 + \psi(s),
$$
and $\psi$ is a $C^2$-potential that satisfies the following:
\begin{enumerate}[(i)]
 \item $\lim_{s \to -1/2} \psi(s) = + \infty$,
 \item $\lim_{s \to 1} \psi(s) = + \infty$,
 \item $\psi'(0) = 0$,
 \item $\psi$ has a minimum at some $s_* \in (0,1)$.
\end{enumerate}
The motivation behind introducing $\psi$ is to confine the variable degree of of orientation $s$ to the range $(-1/2,1)$. 
See  \cite{Ericksen, Maddocks} for a discussion of the Ericksen model in full generality.
We stress that $\psi$ is a $C^2$ function of $s$, which is of lower order than the terms depending on $\nabla s$ or $\nabla n$.
As we are interested in the asymptotic behavior of energy-minimizing configurations near point defects,
we can assume $\psi$ to be the constant potential, since it is subcritical with respect to the scaling of the Ericksen energy.

Under the assumption $\kappa_1 = \kappa_2 = \kappa_3 = 1$, $\kappa_4 = \kappa_6 = 0$, and $\kappa_5 = \kappa$, 
which is often called the {\emph{one constant approximation}},
the Ericksen energy reduces to
\begin{equation}
\int_{\Omega} \left [ \kappa |\nabla s|^2 + s^2 | \nabla n |^2 + \psi(s)  \right ] \, \mathrm{d}x.
\label{Ericksen}
\end{equation}
The key result in \cite{L89, L91} is the following: There exists a map $u = \left ( u_1, u_2 \right ) \, : \, \Omega \to \mathbf{C}_\kappa$ minimizing 
$$
\int_{\Omega} \left [ | \nabla u |^2 + \psi \left ( \kappa^{-1/2} | u | \right ) \right ] \, \mathrm{d}x, 
$$
with respect to Dirichlet boundary data $g \in H^{1/2} \left ( \Omega \right )$, where
$$
\mathbf{C}_\kappa = \left \{ (z,y) \in \mathbf{R} \times \mathbf{R}^3 \, : \, z = \sqrt{\kappa-1} |y| \right \}, \quad \kappa > 1;
$$
the minimizer $u$ is locally H\"older continuous in $\Omega$; and the codimension of $u^{-1} \{ 0 \}$ is at least $2$.
(The existence and regularity results also hold for a double-cone version of $\mathbf{C}_\kappa$, corresponding to $s \in ( -1/2, 1 )$, or the case $\kappa <1$.
However, $u$ can have codimension 1 defects, also known as
{\emph{wall defects}}, in these cases, cf. \cite{AV, HL93}. So we restrict our attention to the case $\kappa > 1$, $s \geq 0$.)
Off the set $u^{-1} \{ 0 \}$, we can define the pair $s = \kappa^{-1/2} |u|$, $n = \sqrt{\kappa} |u|^{-1} u_2$, which minimizes \eqref{Ericksen}.
Moreover, the defect set, $\mathrm{sing}(n) = s^{-1} \{ 0 \} = u^{-1} \{ 0 \}$, cf. \cite[Section 3.2]{L89}.
In other words, up to a lower order perturbation term, the problem of analyzing the minimizers of \eqref{Ericksen}, where $\kappa > 1$ and $s \geq 0$, 
can be recast as the problem of analyzing energy-minimizing harmonic maps into the cone $\mathbf{C}_\kappa$. 

The estimate on the Hausdorff dimension of defects was improved in \cite{HL93} 
via a classification of planar tangent maps at isolated defects in $\mathbf{R}^2$.
In particular, it was proved that there is no nontrivial energy-minimizing map $u \, : \, \mathbf{R}^2 \to \mathbf{C}_\kappa$ such that 
$u^{-1} \{0\} \not = \emptyset$.
Consequently, when the target is $\mathbf{C}_\kappa$, $u^{-1} \{ 0 \}$ is locally discrete in $\Omega \subset \mathbf{R}^3$.
Also in \cite{HL93}, the modified Ericksen model was introduced. In this model, the target $\mathbf{C}_\kappa$ is replaced with its projectivised version
$\mathbf{D}_\kappa$. That is, for $[y] = \left \{ y, -y \right \}$, the sign equivalence class for $y \in \mathbf{R}^3$, 
$$
\mathbf{D}_\kappa = \left \{ \left ( z, [y] \right ) \, : \, (z,y) \in \mathbf{C}_\kappa \right \}, \quad \kappa > 1.
$$
While the results in \cite{L89, L91} hold for the modified Ericksen model as well,
the modified Ericksen model has the additional remarkable feature that in dimension $3$, it admits energy-minimizing configurations with line defects, cf. \cite{HL93, AHL}. 
We also refer to \cite{BZ} for a comprehensive discussion on modeling liquid crystals with line fields and its relation to understanding defects in $Q$-tensor theory. 

More recently, in \cite{AHL}, the structure of defects in energy-minimizing configurations in the modified Ericksen model was analyzed. 
Via blow-up arguments, it was proved that the defect set consists of a finite union of isolated points and H\"older continuous arcs with finitely many crossings.
In \cite{A2018}, this result was strengthened by proving that each of these arcs has finite length, and hence, admits a Lipschitz parametrization.  
It was also proved in \cite{AHL} that at all but countably many points in the defect set, the homogeneous blow-ups of energy-minimizing maps depend on two variables only.
Hence, at any such point, homogeneous blow-ups are planar tangent maps with respect to a suitable coordinate system. 
This coordinate system may a priori depend on the blow-up sequence, and its uniqueness is an interesting open problem.

Finally, we remark that the configurations \eqref{GeneralForm}, which are precisely the planar tangent maps by the classification result in \cite{HL93}, were the subject of investigation in \cite{INSZ} as well, where their stability was proved within the context of Landau-de Gennes model in $Q$-tensor theory.  

\subsection{The Main Result}
We denote a ball with center $z$ and radius $r$ in $\mathbf{R}^2$ by $B_r(z)$.
For maps $u \, : \, B_2(0) \to \mathbf{D}_\kappa$ minimizing the Dirichlet energy, we recall the notion of homogeneous blow-up limits.
\begin{definition}
We call $\phi \, : \, \mathbf{R}^2 \to \mathbf{D}_\kappa$ a homogeneous blow-up limit of $u$ at $x_0 \in \mathbf{R}^2$, if there exists a sequence of positive real numbers $\rho_i \to 0$, for which
\begin{equation*}
\phi(x) = \lim_{\rho_i \to 0}  \left [ \left \| u \left ( x_0 + \rho_i \cdot \right ) \right \|_{L^2 \left ( B_1 ( 0 ) \right )}^{-1} u \left ( x_0 + \rho_i x \right ) \right ].
\end{equation*}
\end{definition}
The mode of convergence was established in \cite{L89, L91, AHL}.
In particular, we have strong convergence in $H^1_{loc} \left ( \mathbf{R}^2 \right )$ and $C^{0,\alpha}_{loc} \left ( \mathbf{R}^2 \right )$ for some $\alpha > 0$, as well as convergence in higher-order norms away from the defect point $x_0$, due to the regularity theory.
We also recall the classification result in \cite{HL93} for homogeneous blow-up limits at the point defects of energy-minimizing maps in dimension 2:
\begin{equation}
 \phi(r, \theta ) = r^{\frac{1}{2 \sqrt{\kappa}}} h(\theta) = 
 C_\kappa r^{\frac{1}{2 \sqrt{\kappa}}} \left ( \sqrt{\kappa-1} , \left [ \Theta \left ( e^{\pm i\theta/2} , 0 \right ) \right ] \right ),
 \label{GeneralForm}
\end{equation}
where $\Theta \in SO(3)$, the rotation group in $\mathbf{R}^3$, $[y] \in \mathbf{RP}^2$ for $y \in S^2$, and $C_\kappa = \sqrt{\frac{2\sqrt{\kappa}+1}{2 \pi \kappa \sqrt{\kappa}}}$ so that
$\| \phi \|_{L^2 \left (B_1(0) \right )} = 1$.

It is not a priori clear whether there can be two different sequences of scales,  
giving rise to distinct blow-up limits $\phi$ and $\tilde{\phi}$, corresponding to distinct $\Theta$, $\tilde{\Theta} \in SO(3)$.
We will rule out this possibility as a byproduct of the following estimate.

\begin{theorem} \label{MainTheorem}
Let $u \, : \, B_2(0) \to \mathbf{D}_\kappa$ be an energy-minimizing harmonic map such that $u^{-1} \{ 0 \} = \{ 0 \}$, and $\phi$, a homogeneous blow-up limit of $u$ at $0 \in \mathbf{R}^2$.
Then there exist $r_0 > 0$, $C > 0$, $\mu \in (0,1)$ such that the estimate,
\begin{equation}
\left | u(x) - A \left (0, r_0 \right ) \phi(x) \right | \leq C A \left (0, r_0 \right ) |x|^{\frac{1}{2\sqrt{\kappa}}+\mu},
\label{GeometricDecay}
\end{equation}
holds for every $|x| \leq r_0$, where
\begin{equation*}
A  \left (0, r_0 \right ) = \left ( \frac{1}{{r_0}^{2+\frac{1}{\sqrt{\kappa}}}} \int_{B_{r_0} (0)} |u|^2 \, \mathrm{d}x \right )^{1/2}.
\end{equation*}
\end{theorem}

\begin{remark} \label{Monotonicity}
The quantity $A(0, r)$ is monotone increasing in $r$. This follows from
the first variation formulas \eqref{Stationarity}, \eqref{Subharmonicity}, and that $N (0, r ) \geq N \left ( 0, 0^+ \right ) = 1/2\sqrt{\kappa}$, 
where $N(x,r)$ is the Almgren frequency function, originally introduced in \cite{A} and defined as
$$
N(x,r) = \frac{r \int_{B_r(x)} | \nabla u |^2 \, \mathrm{d}x}{\int_{\partial B_r(x)} |u|^2 \, \mathrm{d}S}.
$$
\end{remark}

An immediate consequence of the estimate \eqref{GeometricDecay} is the uniqueness of homogeneous blow-up limits.

\begin{corollary}
Suppose $x_0 \in u^{-1} \{0 \}$ for $u \, : \, \Omega \to \mathbf{D}_\kappa$ an energy-minimizing harmonic map, 
where $\Omega \subset \mathbf{R}^2$. 
There exists a $\phi \, : \, \mathbf{R}^2 \to \mathbf{D}_\kappa$, defined as in \eqref{GeneralForm}, which is the unique homogeneous blow-up limit of $u$ at $x_0$, and which satisfies 
\eqref{GeometricDecay}.
\end{corollary}

\begin{proof}
Translating our coordinate system by $x_0$, we can assume without loss of generality that $x_0 = 0$.
Choosing $r_0 = r$ sufficiently small, dividing both sides of \eqref{GeometricDecay} by 
$\left \| u \left ( r \cdot \right ) \right \|_{L^2 \left ( B_1 \left ( 0 \right ) \right )}$,
and performing the change of coordinates, $x = r y$, 
we observe that for every $r$ sufficiently small and every $y \in B_1(0)$,
\begin{equation*}
\left | \left \| u \left ( r \cdot \right ) \right \|_{L^2 \left ( B_1 \left ( 0 \right ) \right )}^{-1} 
u \left ( r y \right ) - \phi(y) \right | \leq C |y|^{\frac{1}{2\sqrt{\kappa}} + \mu} r^{\mu}.
\end{equation*}
Hence, the limit of left-hand side is $0$, as $r \downarrow 0$. In other words, the homogeneous blow-up limit of $u$ at $x_0$ is unique.
\end{proof}

\subsection{An Overview of the Proof}

Our strategy for the proof of Theorem \ref{MainTheorem} goes back to the work of Allard and Almgren \cite{AA} on the uniqueness of tangent cones at isolated singularities of minimal surfaces. 
The same idea later found applications in the context of Dirichlet energy-minimizing harmonic maps into analytic targets \cite{SL} and maps into $S^{n-1}$ minimizing the $p$-energy \cite{HLW}.
This strategy relies on the integrability of the second variation equation in order to construct suitable comparison maps,
while making use of an appropriate monotonicity formula to knock out coefficients of non-decaying terms in a certain eigenfunction expansion.
Here the Weiss monotonicity formula \eqref{Weiss0}, which was originally introduced in \cite{Weiss} in the context of free boundary problems,
is the appropriate monotonicity formula that enables us to carry out a blow-up argument as in \cite{AA, SL, HLW}.

We will prove Theorem \ref{MainTheorem} through proving a series of lemmas. But firstly, we define for $0 < \rho < \sigma \leq 1$, $u_r(x) = r^{-1/2\sqrt{\kappa}} u(x)$, 
the following weighted $C^2$-norm,
\begin{equation*}
\| u \|_{\sigma, \rho} = \sup_{\rho \leq r \leq \sigma} \left ( \left  \| u_r \right \|_{L^\infty\left ( \partial B_r(0) \right )}  
+  r \left \| \nabla u_r \right \|_{L^\infty\left ( \partial B_r(0) \right )} 
+  r^2 \left \| \nabla^2 u_r \right \|_{L^\infty\left ( \partial B_r(0) \right )} \right ).
\end{equation*}

For any $h(\theta) = C_\kappa \left ( \sqrt{\kappa - 1}, \left [ \Theta \left (e^{\pm i \theta/2}, 0 \right ) \right ] \right )$, 
$\epsilon > 0$, and $\lambda \in (0,1/2)$, we define the set $\mathcal{Q}(h,\epsilon,\lambda)$
of all minimizing harmonic maps $u \, : \, B_1(0) \to \mathbf{D}_\kappa$ such that $u^{-1} \{ 0 \} = \{ 0 \}$, and
\begin{equation*}
\left \| u(r,\theta) - A_0 r^{1/2 \sqrt{\kappa}} h(\theta) \right \|_{1, \lambda^3} \leq \epsilon,
\end{equation*}
where 
\begin{equation}
A_0 = \| u \|_{L^2 \left ( B_1(0) \right )}. \label{AChoice}
\end{equation} 

The following decay lemma is the key ingredient in proving Theorem \ref{MainTheorem}.

\begin{lemma} \label{KeyLemma}
For any positive $\lambda \leq \lambda_0$, where $\lambda_0 \in (0,1)$ is an absolute constant, 
there exists an $\epsilon > 0$, such that if $h_1 \in C^2 \left ( S^1, \mathbf{D}_\kappa \right )$
is a map defined as in \eqref{GeneralForm}, corresponding to a $\Theta_1 \in SO(3)$, satisfying
$A_0 \left \| h_1 - h \right \|_{C^2 \left ( S^1 \right ) } < \epsilon $, and $u \in \mathcal{Q} (h, \epsilon, \lambda )$, 
then there exists a $\Theta_2 \in SO(3)$ and a corresponding map $h_2 \in C^2 \left (S^1, \mathbf{D}_\kappa \right )$ such that we have:
\begin{equation}
\left \| u(r,\theta) - A_0 r^{1/2\sqrt{\kappa}} h_2(\theta) \right \|_{\lambda, \lambda^3} 
\leq \frac{1}{2} \left \| u(r,\theta) - A_0 r^{1/2\sqrt{\kappa}} h_1(\theta) \right \|_{1, \lambda^2}.
\label{Halving}
\end{equation}
\end{lemma}

Applying Lemma \ref{KeyLemma} repeatedly, we will prove the following result.

\begin{lemma} \label{GeometricEstimate}
There are positive constants $\lambda_0$, $\epsilon_0$, $\eta$, $C$ depending on $\kappa$ and $\Theta$ only such that 
if $\lambda \in \left ( 0 , \lambda_0 \right ]$, $\epsilon \in \left ( 0, \epsilon_0 \right ]$, $h \in C^2 \left (S^1, \mathbf{D}_\kappa \right )$ is a map as in
\eqref{GeneralForm} with $\Theta \in SO(3)$, and
$u \in \mathcal{Q} (h, \epsilon, \lambda )$ satisfies the estimate,
\begin{equation}
\left \| u(r, \theta ) - A_0 r^{1/2\sqrt{\kappa}} h(\theta) \right \|_{1, \lambda^2} < \eta \epsilon, \label{Initial}
\end{equation}
then there exists a $\hat{\Theta} \in SO(3)$ and a corresponding map 
$\hat{h} \in C^2 \left ( S^1, \mathbf{D}_\kappa \right )$ as in \eqref{GeneralForm}, and a positive constant $\mu = \mu ( \lambda )$ such that for every $r \in [0,1]$,
\begin{equation}
\left \| u(r, \theta ) -  A_0 r^{1/2\sqrt{\kappa}} \hat{h}(\theta) \right \|_{C^2 \left ( S^1 \right )} \leq C \eta \epsilon r^{\frac{1}{2 \sqrt{\kappa}}+\mu} \label{CircleEstimate}.
\end{equation}
\end{lemma}

Finally, using Lemma \ref{GeometricEstimate}, we will prove Theorem \ref{MainTheorem}.

Our plan is as follows: 
In Section \ref{IntegrabilitySection}, we introduce Jacobi fields along harmonic maps, their integrability, and its consequences, as exploited in \cite{AA, SL, HLW}.  
In Section \ref{Linearization}, we derive the Jacobi field equations along harmonic maps $u \, : \, \mathbf{R}^2 \to \mathbf{C}_\kappa$ with $u^{-1} \{ 0 \} = \{ 0 \}$,
and prove a classification result, which is a crucial ingredient in proving Lemma \ref{KeyLemma}.
In Section \ref{ExtensionSection}, we prove the important extension result, Lemma \ref{Extension}, which corresponds to Remark 6.2, (iii) in \cite{SL}. 
In Section \ref{KeySection}, we put together the tools introduced in Sections \ref{IntegrabilitySection}, \ref{Linearization} and \ref{ExtensionSection} to prove Lemma \ref{KeyLemma}.
Finally, in Section \ref{FinalSection}, we prove Lemma \ref{GeometricEstimate} and Theorem \ref{MainTheorem}.

\section{Integrability of Jacobi Field Equation and Consequences} \label{IntegrabilitySection}

We recall some definitions and classical results from \cite{AA, SL, LW}.
For a harmonic map $\phi \in C^2 ( M , N )$, the Jacobi field equation along $\phi$ is defined as: 
\begin{equation}
L_\phi ( \psi ) = \left . \frac{d}{ds} \right \vert_{s=0} \tau \left ( \phi_s \right ) = 0 ,
\label{JacobiDefinition}
\end{equation}
where $\tau \left ( \phi_s \right )$ is the tension field for the $C^1$-family of deformation maps $\phi_s \in C^1 \left ( ( - 1 , 1 ), C^2 (M,N) \right )$ satisfying
$$
\phi_0 = \phi \quad \mathrm{and} \quad \left . \frac{d}{ds} \right \vert_{s=0} \phi_s = \psi.
$$
We say that the Jacobi field equation \eqref{JacobiDefinition} along $\phi$ is integrable, if for any $\psi \in  C^2 \left ( M , \phi^* TN \right )$ solving $L_\phi \psi = 0$,
there exists a $C^1$-family of smooth harmonic maps $\phi_s \in C^1 \left ( ( - 1 , 1 ), C^2 (M,N) \right )$ such that
\begin{equation}
\lim_{s \to 0} \left \| s^{-1} \left ( \phi_s - \phi \right ) - \psi \right \|_{C^2 (M)} = 0. 
\label{JacobiLimit}
\end{equation}

A key observation utilized in \cite{AA, SL} is the following: When $N$ is a real analytic submanifold isometrically embedded in $\mathbf{R}^n$, and 
$\phi \in C^2 \left ( S^{m-1} , N \right )$ is a harmonic map such that
$L_\phi \psi = 0$ is integrable, for $\delta > 0$ suitably small, there is a real-analytic embedding,
$$
\Psi \, : \, \mathrm{ker} L_\phi \cap \left \{ \hat{\psi} \in C^2 \left ( S^{m-1} , \phi^* TN \right ) \, : \, \left \| \hat{\psi} \right \|_{C^2 \left ( S^{m-1} \right )} < \delta \right \}
\to L^2 \left ( S^{m-1} \right ),
$$
and the image of $\Psi$ is a $\mathrm{dim} \left ( \mathrm{ker} \left ( L_\phi \right ) \right)$-dimensional real-analytic manifold $\mathcal{M}$ containing
$$
S_\delta = \left \{ \hat{\phi} \in C^2 \left ( S^{m-1}, N \right ) \, : \, \tau \left ( \hat{\phi} \right ) = 0, \; \left \| \hat{\phi} - \phi \right \|_{C^2 \left ( S^{m-1} \right )} 
< \delta/2 \right \} .
$$

Furthermore, when \eqref{JacobiDefinition} is integrable, $S_\delta$ contains a neighborhood of $\phi$ in $\mathcal{M}$, and consequently,
for $\delta > 0$ sufficiently small, \eqref{JacobiLimit} holds uniformly with $\hat{\phi} \in S_\delta$ in place of $\phi$. More precisely, there is a $\delta_0 > 0$ such that
if $\hat{\phi} \in S_{\delta_0}$ and $L_{\hat{\phi}}\psi = 0$ with $| \psi |_{C^2 \left ( S^{m-1} \right )} < \delta_0$, then there is a $\tilde{\phi} \in S_{2 \delta_0}$ with
\begin{equation}
\tilde{\phi} = \hat{\phi} + \psi + \xi, \quad | \xi |_{C^2 \left ( S^1 \right )} \leq C | \psi |_{C^2 \left ( S^{m-1} \right ) }^2,
\label{Taylor}
\end{equation}
where $C = C ( \phi, N )$.

Finally, we remark that the condition $\tau \left ( \hat{\phi} \right ) = 0$ in the definition of $S_{\delta_0}$
guarantees the constancy of the Dirichlet energy along smooth paths in $S_{\delta_0}$, when $\delta_0$ is sufficiently small.
In other words, for $\delta_0$ sufficiently small, for all $\hat{\phi} \in S_{\delta_0}$, 
\begin{equation}
\int_{S^{m-1}} \left | \frac{d \hat{\phi}}{d \theta} \right |^2 \, \mathrm{d}S =  \int_{S^{m-1}} \left | \frac{d \phi}{d \theta} \right |^2 \, \mathrm{d}S.
\label{EnergyNeighborhood}
\end{equation}

\begin{remark} \label{CircleMaps}
A simple case in which the integrability of the Jacobi field equation holds is that of geodesics, in other words, harmonic maps depending on a single variable.
It is a classical result in differential geometry that for any Jacobi field $J$ along a geodesic $\gamma$, there exists a one-parameter family of geodesics $\Phi_t$
such that $\Phi_0 = \gamma$ and $\left . \frac{d}{dt} \Phi_t \right \vert_{t=0} = J$, cf. \cite[Chapter 5]{doCarmo}.

In particular, for any Jacobi field $J$ along the equator $n(\theta) = \left ( \cos(\theta), \sin(\theta), 0 \right )$, interpreted as a harmonic map from $S^1$ to $S^2$,
we have a family of harmonic maps $\Phi_t$ from $S^1$ to $S^2$ such that $\Phi_0 = n$ and $\left . \frac{d}{dt} \Phi_t \right \vert_{t=0} = J$.
This simple observation, combined with the classification result of Section \ref{Linearization}, will play a crucial role in proving the key geometric decay estimate in Section \ref{KeySection}.
\end{remark}

\section{Euler-Lagrange and Jacobi Field Equations for Maps into $\mathbf{C}_\kappa$} \label{Linearization}

We derive the Jacobi field equations along harmonic maps $u \, : \, \mathbf{R}^2 \to \mathbf{C}_\kappa$ with $u^{-1} \{ 0 \} = \{ 0 \}$,
and prove a classification result, which is a crucial ingredient in proving Lemma \ref{KeyLemma}. 

Away from $0 \in \mathbf{R}^2$, we can define $s = 1/\sqrt{\kappa} |u|$ and $n \, : \, \mathbf{R}^2 \to \mathbf{S}^2$ such that $u = \left ( \sqrt{\kappa-1}s, sn \right )$.
By the regularity theory (cf. \cite{AHL}), $s$ and $n$ are locally smooth in $\mathbf{R}^2 \backslash \{0 \}$. By considering the first variations of the target, 
we obtain the Euler-Lagrange equations satisfied by $s$ and $n$ in $\mathbf{R}^2 \backslash \{0 \}$:
\begin{subequations}
\begin{align}
\kappa \Delta s - s | \nabla n |^2 = 0, \label{eqndegree} \\
s^2 \Delta n + 2 s \nabla s \cdot \nabla n + s^2|\nabla n |^2 n = 0. \label{eqndirector}
\end{align}
\end{subequations}

A straightforward calculation yields that the coupled equations \eqref{eqndegree} and \eqref{eqndirector} can be expressed in terms of $u$ as:
\begin{equation}
\Delta u + N_\kappa \left ( u, \nabla u \right ) u = 0, 
\label{fullLagrange}
\end{equation}
where the matrix $N_\kappa \left ( u, \nabla u \right )$ is given by
\begin{equation}
N_\kappa \left ( u, \nabla u \right ) =
\left |
\nabla \left ( |u|^{-1} u \right )
\right |^2 \mathcal{J}_\kappa
=
\left |
\nabla \left ( |u|^{-1} u \right )
\right |^2
\left (
\begin{matrix}
-1 & 0 \\
0 & \kappa -1 
\end{matrix}
\right ) .
\end{equation}

Linearizing \eqref{fullLagrange} around the solution $u(r,\theta) = r^{1/\sqrt{\kappa}} \left ( \sqrt{\kappa-1}, \left ( e^{i \theta}, 0 \right ) \right )$ of \eqref{fullLagrange},
we obtain the Jacobi field equations for $\psi \, : \, \mathbf{R}^2 \backslash \{0\} \to u^* T \mathbf{C}_\kappa$:
\begin{equation}
\begin{aligned}
\Delta \psi & +  2 \left [ \left \langle \nabla \left ( |u|^{-1} u \right ) \, , \, \nabla \left ( |u|^{-1} \psi \right ) \right \rangle 
+ \left \langle \nabla \left ( |u|^{-1} u \right ) \, , \, \nabla \left ( - |u|^{-3} \left ( u \cdot \psi \right ) u \right ) \right \rangle \right ] 
\mathcal{J}_\kappa u \\
& + \left | \nabla \left ( |u|^{-1} u \right ) \right |^2
\mathcal{J}_\kappa \psi = 0.
\label{JacobiFull}
\end{aligned}
\end{equation}
Letting $n(\theta) = \left ( \cos \left ( \theta \right ), \sin \left ( \theta \right ) ,0 \right )$ and
$\psi = \left ( \psi_1, \psi_2 \right )$, \eqref{JacobiFull} reduces to
\begin{equation}
\begin{aligned}
\frac{1}{r} \partial_r \left ( r \partial_r \psi \right ) +
\frac{1}{r^2} \mathcal{L}_u (\psi)  = 0, \label{JacobiCompact}
\end{aligned}
\end{equation}
where
\begin{equation*}
\mathcal{L}_u (\psi) = \partial_\theta^2 \psi + \frac{2}{\kappa} \left [ \left \langle \partial_\theta n , \partial_\theta \psi_2 \right \rangle 
+ \frac{\left ( \sqrt{\kappa - 1} \psi_1 + n \cdot \psi_2 \right )}{\kappa} \right ] \mathcal{J}_\kappa \left  ( \sqrt{\kappa-1}, n \right ) + \frac{1}{ \kappa} \mathcal{J}_\kappa \psi .
\end{equation*}
Note that the only geometrically relevant solutions $\psi$, that is $\psi \, : \, \mathbf{R}^2 \backslash \{ 0 \} \to u^*T \mathbf{C}_\kappa$, 
are those satisfying the orthogonality relation 
$\psi \cdot \mathcal{J}_{\kappa}u = 0$, which is equivalent to $\psi_1 = \sqrt{\kappa-1} \left ( n \cdot \psi_2 \right )$.
Hence, for such $\psi$, we can rewrite $\mathcal{L}_u (\psi)$ as:
\begin{equation}
\mathcal{L}_u (\psi) = \partial_\theta^2 \psi + \frac{2}{\kappa} \left [ \left \langle \partial_\theta n , \partial_\theta \psi_2 \right \rangle 
+ \left ( n \cdot \psi_2 \right ) \right ] \mathcal{J}_\kappa \left  ( \sqrt{\kappa-1}, n \right ) + \frac{1}{ \kappa} \mathcal{J}_\kappa \psi .
\label{CircleOperator}
\end{equation}

By a standard second variation calculation, the Jacobi operator in \eqref{JacobiFull} is self-adjoint. Since the first term in \eqref{JacobiCompact} is clearly self-adjoint as well,
we conclude that $\mathcal{L}_u$ is a self-adjoint operator on $L^2 \left ( S^1 \right )$. 

For $\psi(r,\theta) = r^{1/\sqrt{\kappa}} \tilde{\psi}(r, \theta)$ solving \eqref{JacobiCompact}, $\tilde{\psi} = \left ( \tilde{\psi_1}, \tilde{\psi_2} \right )$ 
satisfies the following equation on $\mathbf{R}^2 \backslash \{ 0 \}$:
\begin{equation}
 \partial_r^2 \tilde{\psi} + \left ( 1 + \frac{2}{\sqrt{\kappa}} \right ) \frac{1}{r} \partial_r \tilde{\psi} +
+ \frac{1}{r^2} \mathcal{L}'(\tilde{\psi}) = 0,
\label{JacobiFullMod}
\end{equation}
where
\begin{equation}
\begin{aligned}
\mathcal{L}' \left (\tilde{\psi} \right ) = \partial_\theta^2 \tilde{\psi} + \frac{1}{ \kappa} \tilde{\psi} & + 
\frac{2}{\kappa} \left [ n^\perp \cdot \partial_\theta \tilde{\psi}_2 + \left ( n \cdot \tilde{\psi_2} \right ) \right ] \mathcal{J}_\kappa \left ( \sqrt{\kappa - 1}, n \right ) \\
& + \frac{1}{ \kappa} \mathcal{J}_\kappa \left ( \sqrt{\kappa-1} \left ( \tilde{\psi_2} \cdot n \right ) ,  \tilde{\psi_2} \right ),
\label{CircleOperatorMod}
\end{aligned}
\end{equation}
and $n^\perp = \partial_\theta n$. 

An important observation is that $\mathcal{L}' \left ( \tilde{\psi} \right )$ differs from $\mathcal{L} \left ( \tilde{\psi} \right )$ 
by the term $\frac{1}{ \kappa} \tilde{\psi}$ only.
Therefore, $\mathcal{L}'$ is a self-adjoint, elliptic operator on $L^2 \left ( S^1 \right )$ as well. 
In particular, it has eigenvalues $\lambda_1 \leq \lambda_2 \leq ... \lambda_j \leq ...$, where $\lambda_j \to + \infty$, as $j \to + \infty$,
and admits a corresponding orthonormal eigenbasis $\left \{ \tilde{\Psi}_j \right \}$ of $L^2 \left ( S^1, u^* T \mathbf{C}_\kappa \right )$.

Under the orthogonality assumption above, we can find smooth functions $f, g, h$ such that $\tilde{\psi_2} = f n + g n^\perp + h e_3$, where $e_3 = (0,0,1)$. 
By a straightforward calculation, $\mathcal{L}' \left ( \tilde{\psi} \right ) = 0$ reduces to the following system:
\begin{equation}
\begin{aligned}
\kappa f'' - 4f = 2g', \\
\kappa f'' + 4 (\kappa-1)f = 2g', \\ 
g'' + 2f' = 0, \\
h'' + h   = 0.
\label{ReducedSystem}
\end{aligned}
\end{equation}
Since $\kappa > 1$, the first two equations in \eqref{ReducedSystem} imply $f \equiv 0$.
Hence, from $f \equiv 0$ and the first equation, we deduce $g \equiv C$.
Finally, from the last equation, we deduce $h(\theta) = A \cos ( \theta) + B \sin (\theta )$, 
and conclude that $\tilde{\psi} = \left ( 0, C n^\perp + h e_3 \right )$.

Note that the solutions $\tilde{\psi}$ are contained in the set of Jacobi fields along $n$, the equator of $S^2$.
As recalled in Remark \ref{CircleMaps}, the Jacobi fields along any geodesic is integrable.
In other words, for every $\tilde{\psi}$ solving $\mathcal{L}' \left ( \tilde{\psi} \right ) = 0$, there exists a one-parameter family of geodesics, 
$$
\Phi \; : \; \left ( - \epsilon, \epsilon \right ) \times \left [ 0 , 2\pi \right ] \to S^2,
$$
such that $\Phi(0,\theta) = n(\theta)$ and $\left . \frac{d}{dt} \Phi(t,\theta) \right \vert_{t=0} = \tilde{\psi}$.

Consequently, for any $\tilde{\psi}$ solving $\mathcal{L}' \left ( \tilde{\psi} \right ) = 0$ under the above orthogonality relations,
there exists a $C^1$-family of locally smooth harmonic maps $\{ u_t \}$ from $\mathbf{R}^2 \backslash \{ 0 \}$ into $\mathbf{C}_\kappa$,
\begin{equation*}
u_t (r,\theta) = C(\kappa) \left ( \sqrt{\kappa-1} r^{1/\sqrt{\kappa}} , r^{1/\sqrt{\kappa}} \Phi \left ( t, \theta \right ) \right ),
\end{equation*}
such that $ u_0 = u$ and $\left . \frac{d}{dt} u_t \right \vert_{t=0} = r^{1/\sqrt{\kappa}} \tilde{\psi} = \psi$.
In fact, given $\tilde{\psi}(\theta) = C n^\perp(\theta) + A \cos(\theta) + B \sin(\theta)$, we can set
$\Phi(t,\theta) = M_1(Bt) M_2(At) M_3(Ct) n(\theta)$, where
$M_i(\tau)$ is the matrix of counter-clockwise rotation of angle $\tau$ around the $e_i$-axis in $\mathbf{R}^3$ for $i=1,2,3$.

Similarly, we can linearize \eqref{fullLagrange} 
around the solution 
$$u(r,\theta) = r^{1/\sqrt{\kappa}} \left ( \sqrt{\kappa-1}, \left ( e^{- i \theta}, 0 \right ) \right )$$ 
of \eqref{fullLagrange}, and obtain an analogous result by following the above argument step-by-step. Hence, we omit the details.

\section{An Extension Property} \label{ExtensionSection}

We prove the following extension property for $ \mathcal{Q} ( h, \epsilon, \lambda ) $, the set 
of energy-minimizing harmonic maps $u \, : \, B_1(0) \to \mathbf{D}_\kappa$ such that $u^{-1} \{ 0 \} = \{ 0 \}$, and
$$
\left \| u(r, \theta) - A_0 r^{1/2\sqrt{\kappa}} h(\theta) \right \|_{1, \lambda^3} \leq \epsilon,
$$
where $h(\theta) = C_\kappa \left ( \sqrt{\kappa - 1}, \left [ \Theta \left ( e^{\pm i\theta/2},0 \right ) \right ] \right )$.
Our proof is based on a contradiction argument and the Weiss monotonicity formula \eqref{Weiss0}.

\begin{lemma} \label{Extension}
For any $h(\theta) = C_\kappa \left ( \sqrt{\kappa - 1}, \left [ \Theta \left (e^{\pm i \theta/2}, 0 \right ) \right ] \right )$, $\lambda > 0$ and $\mu > 0$,
there is a positive $\hat{\epsilon} = \hat{\epsilon}(\mu, h) \in (0,1)$ such that
\begin{equation}
\mathcal{Q} \left ( h, \hat{\epsilon}, \lambda \right ) \subset \mathcal{Q} \left ( h, 1, \mu \lambda \right ).
\label{Inclusion}
\end{equation}
\end{lemma}

\begin{proof}
If $u \in \mathcal{Q} \left ( h, \hat{\epsilon}, \lambda \right )$, then
$$
v(x) = \left ( 8 \lambda^3 \right )^{-1/2\sqrt{\kappa}} u \left ( 8 \lambda^3 x \right )
\in \mathcal{Q} \left ( h, \hat{\epsilon}, \frac{1}{2} \right ).
$$
Hence, verifying $v \in \mathcal{Q} \left ( h, 1, \frac{\mu}{2} \right )$ is equivalent
to verifying $u \in \mathcal{Q} \left ( h, 1, \mu \lambda \right )$. Thus, we may assume $\lambda = \frac{1}{2}$. 

Suppose \eqref{Inclusion} does not hold.
Then there exists an $h$ defined as in the hypothesis, $\mu \in (0,1)$, and energy-minimizing harmonic maps $u_i \, : \, B_2(0) \to \mathbf{D}_\kappa$,
and $\epsilon_i \downarrow 0$ such that
\begin{equation}
u_i \in \mathcal{Q} \left ( h, \epsilon_i, \frac{1}{2} \right ) \backslash \mathcal{Q}  \left ( h, 1, \frac{\mu}{2} \right ). \label{ContraHyp}
\end{equation}

We recall that considering the appropriate domain and target variations for $u_i$, (cf. \cite[Lemma 2.1]{AHL}), 
we get the respective monotonicity formulas:
\begin{equation}
\frac{d}{dr} \left ( \int_{B_r(0)} \left | \nabla u_i \right |^2 \, \mathrm{d}x \right ) = 2 \int_{\partial B_r(0)} \left | \frac{\partial u_i}{\partial r} \right |^2 \, \mathrm{d}S,
\label{Stationarity}
\end{equation}
\begin{equation}
\frac{d}{dr} \left ( \int_{\partial B_r(0)} \left | u_i \right |^2 \, \mathrm{d}S \right ) =
\frac{1}{r} \int_{\partial B_r(0)} \left | u_i \right |^2 \, \mathrm{d}S + 2 \int_{B_r(0)} \left | \nabla u_i \right |^2 \, \mathrm{d}x,
\label{Subharmonicity}
\end{equation}
for almost every $r \in (0,2)$. 
Using these identities, integrating by parts, and expressing the three resulting terms in the integrand as a square, 
we arrive at the well-known Weiss monotonicity formula,
\begin{equation}
\begin{aligned}
\frac{d}{dr} \left [ \frac{1}{r^{1/\sqrt{\kappa}}} \int_{B_r(0)} \left | \nabla u_i \right |^2 \, \mathrm{d}x 
- \frac{1/ 2 \sqrt{\kappa}}{r^{1+ 1/\sqrt{\kappa}}} \int_{\partial B_r(0)} \left | u_i \right |^2 \, \mathrm{d}S \right ] = \\
2 \int_{\partial B_r(0)} \left | \frac{\partial \left ( u_i / r^{1/2\sqrt{\kappa}} \right ) }{\partial r} \right |^2 \, \mathrm{d}S,  
\end{aligned}
\label{WeissDiff}
\end{equation}
for almost every $r \in (0,2)$. Note that the right-hand side is $0$, if and only if $u_i$ is homogeneous of degree $1/2\sqrt{\kappa}$.

Integrating the Weiss monotonicity formula on $[0,1]$ gives
\begin{equation}
\begin{aligned}
2 \int_{B_1(0)} \left | \frac{\partial \left ( u_i/ r^{1/2\sqrt{\kappa}} \right )}{\partial r} \right |^2 \, \mathrm{d}x =
   \int_{B_1(0)} \left | \nabla u_i \right |^2 - \frac{1}{2 \sqrt{\kappa}} \int_{\partial B_1(0)} \left | u_i \right |^2 \, \mathrm{d}S.
\label{Weiss0}
\end{aligned}
\end{equation}
For $\phi(r,\theta) = r^{1/2\sqrt{\kappa}} h(\theta)$, by the homogeneity of $\phi$, we have:
\begin{equation*}
\begin{aligned}
2 \int_{B_1(0)} \left | \frac{\partial \left ( u_i/ r^{1/2\sqrt{\kappa}} \right )}{\partial r} \right |^2 \, \mathrm{d}x =
 \int_{B_1(0)} \left | \nabla u_i \right |^2 \, \mathrm{d}x - \frac{1}{2 \sqrt{\kappa}} \int_{\partial B_1(0)} \left | u_i \right |^2 \, \mathrm{d}S - \\
 A_0^2 \left [ \int_{B_1(0)} \left | \nabla \phi \right |^2 \mathrm{d}x - \frac{1}{2 \sqrt{\kappa}} \int_{\partial B_1(0)} \left | \phi \right |^2 \, \mathrm{d}S \right ] .
\end{aligned}
\end{equation*}
By the minimality of $u_i$,
$
\int_{B_1(0)} \left | \nabla u_i \right |^2 \, \mathrm{d}x
\leq \int_{B_1(0)} \left | \nabla \left ( |x|^{1/2\sqrt{\kappa}} u_i \left ( |x|^{-1} x \right ) \right ) \right |^2 \mathrm{d}x.
$
Consequently,
\begin{equation*}
\begin{aligned}
2 \int_{B_1(0)} \left | \frac{\partial \left ( u_i/ r^{1/2\sqrt{\kappa}} \right )}{\partial r} \right |^2 \, \mathrm{d}x & \leq
 \sqrt{\kappa} \left [ \int_{S^1} \left | \partial_\theta u_i \right |^2(1,\theta) \, \mathrm{d} \theta - A_0^2 \int_{S^1} \left | \partial_\theta \phi \right |^2(1,\theta) \, \mathrm{d} \theta \right ]
 \\ & + \frac{1}{4\sqrt{\kappa}} \left  [ A_0^2 \int_{S^1} \left [ h \right |^2(\theta) \, \mathrm{d} \theta - \int_{S^1} \left |u_i \right |^2(1,\theta) \, \mathrm{d} \theta \right ] \\
 & \leq \sqrt{\kappa} \left [ \int_{S^1} \left | \partial_\theta u_i \right |^2(1,\theta) \, \mathrm{d} \theta - A_0^2 \int_{S^1} \left | \partial_\theta \phi \right |^2(1,\theta) \, \mathrm{d} \theta \right ],
\end{aligned}
\end{equation*}
since by Remark \ref{Monotonicity}, the normalization $ \int_{B_1(0)} |\phi|^2 \, \mathrm{d}x = 1$ and \eqref{AChoice},  
\begin{equation*}
\begin{aligned}
A_0^2 & = \int_{B_1(0)} \left | u_i \right |^2 \, \mathrm{d}x \\
& = \int_0^1 \left ( \frac{1}{r^{1+1/ \sqrt{\kappa}}} \int_{\partial B_r(0)} \left | u_i \right |^2 \, \mathrm{d}S \right ) r^{1+1/\sqrt{\kappa}} \, \mathrm{d}r \\
& \leq \left [ \int_0^1 r^{1+1/ \sqrt{\kappa}} \, \mathrm{d}r \right ] \int_0^{2\pi} |u_i |^2 \, \mathrm{d} \theta
= \left [ \int_0^{2\pi} |h|^2 \, \mathrm{d}\theta \right ]^{-1} \int_0^{2\pi} |u_i |^2 \, \mathrm{d} \theta.
\end{aligned}
\end{equation*}
We also claim that for $i$ large enough,
\begin{equation*}
\int_{S^1} \left | \partial_\theta u_i \right |^2(1,\theta) \, \mathrm{d} \theta - A_0^2 \int_{S^1} \left | \partial_\theta \phi \right |^2(1,\theta) \, \mathrm{d} \theta
\leq \left \| u_i(1,\theta) - A_0 \phi(1,\theta) \right \|_{C^2 \left ( S^1 \right )}^2.
\end{equation*}
Note that $h(\theta) = \left ( \sqrt{\kappa-1}, h_2 \right )$, where $h_2$ is a harmonic map from $S^1$ to $\mathbf{RP}^2$, and the Dirichlet energy on $S^1$ is uniformly
$C^2$ in a $C^2$-neighborhood of constrained maps near $h_2$. We expand the Dirichlet energy at $h_2$, and observe that the first order term vanishes, due to the harmonicity of $h_2$.
Therefore, the claim holds.

Hence, we obtain the key estimate,
\begin{equation}
\int_{B_1(0)} \left | \frac{\partial \left ( u_i/ r^{1/2\sqrt{\kappa}} \right )}{\partial r} \right |^2 \, \mathrm{d}x \leq C \left \| u_i - \phi \right \|_{1, 2^{-3}}^2.
\label{Weiss}
\end{equation}

Defining $u_i \left ( \sigma \right ) \left ( \theta \right ) = \sigma^{-1/2\sqrt{\kappa}} u_i \left ( \sigma, \theta \right ) \, : \, S^{1} \to \mathbf{D}_\kappa$,
by the Minkowski inequality, the Cauchy-Schwarz inequality, and \eqref{Weiss}, we have:
\begin{equation}
\begin{aligned}
\left \| u_i ( \tau ) - u_i (1) \right \|_{L^2 \left ( S^1 \right ) }
& \leq \int_\tau^1 \left \| \frac{\partial \left ( u_i / r^{1/2\sqrt{\kappa}} \right )}{\partial r} \right \|_{L^2 \left ( S^1 \right ) } \, \mathrm{d}r \\
& \leq  \left ( \int_\tau^1 \frac{1}{r} \, \mathrm{d}r \right )^{1/2} 
\left ( \int_{B_1(0)} \left | \frac{\partial \left ( u_i / r^{1/2\sqrt{\kappa}} \right ) }{\partial r} \right |^2 \, \mathrm{d}x \right )^{1/2} \\
& \leq \left | \log \tau \right |^{1/2} \left \| u_i - A_0 \phi \right \|_{1, 2^{-3}} \leq \left | \log \tau  \right |^{1/2} \epsilon_i.
\end{aligned}
\label{RingEstimateP}
\end{equation}

Hence, for any $\eta > 0$, choosing $i$ sufficiently large so that $$\max \left \{ \epsilon_i, \epsilon_i \left | \log \left ( \mu / 4 \right ) \right |^{1/2}  \right \} \leq \eta/4,$$
by the triangle inequality, \eqref{RingEstimateP} and \eqref{ContraHyp}, we obtain:
\begin{equation}
\begin{aligned}
\left \| u_i( \tau ) - A_0 \phi \right \|_{L^2 \left ( S^1 \right ) } & \leq 
\left \| u_i \left ( 1 \right ) - A_0 \phi \right \|_{L^2 \left ( S^1 \right ) } +
\left \| u_i ( \tau  ) - u_i \left ( 1 \right )  \right \|_{L^2 \left ( S^1 \right ) } \\ & \leq
\left ( 1 + \left | \log \left ( \mu / 4 \right ) \right |^{1/2} \right ) \epsilon_i \leq \frac{\eta}{2},
\end{aligned}
\label{RingEstimate}
\end{equation}
for all $\mu/4 \leq \tau \leq 1$.

Note that as $\mu$ is fixed, we work uniformly away from $0 \in \mathbf{R}^2$. Recalling the definition 
$u_i \left ( \tau \right ) \left ( \theta \right ) = \tau^{-1/2\sqrt{\kappa}} u_i \left ( \tau, \theta \right ) $,
we can adapt the classical regularity theory for harmonic maps (cf. \cite{SU}) in order to conclude that given $\eta$ small enough, \eqref{RingEstimate} implies
\begin{equation*}
\left \| u_i - A_0 \phi \right \|_{1, \mu/2} \leq 1,
\end{equation*}
which contradicts \eqref{ContraHyp}.
\end{proof}

Strictly speaking, instead of Lemma \ref{Extension}, we will apply an immediate corollary of it in proving Lemma \ref{KeyLemma}.
In order to state this corollary, firstly we define $\mathcal{Q}' (h, \epsilon, \lambda )$ 
as the set of maps $\bar{u} = u \circ \zeta$, where $u \in \mathcal{Q} (h, \epsilon, \lambda )$ and $\zeta(r, \theta) = \left ( r^2, 2 \theta \right )$.
We remark that if $\overline{u} \in \mathcal{Q}' (h, \epsilon, \lambda )$ and $\overline{h} = h \circ \zeta$, then the following estimate holds:
\begin{equation*}
\left \| \overline{u} ( \rho, \omega ) - A_0 \rho^{1/\sqrt{\kappa}} \overline{h}( \omega) \right \|_{1, \lambda^{3}} \leq 4 \epsilon,
\end{equation*}
once we modify the definition of the norm $\| w \|_{\sigma,\rho}$ by defining $w_r(x) = r^{-1/\sqrt{\kappa}} w(x)$, as well as
\begin{equation}
\| w \|_{\sigma,\rho} = \sup_{\rho^2 \leq r \leq \sigma^2} \left ( \left  \| w_r \right \|_{L^\infty\left ( \partial B_r(0) \right )}  
+  r \left \| \nabla w_r \right \|_{L^\infty\left ( \partial B_r(0) \right )} 
+  r^2 \left \| \nabla^2 w_r \right \|_{L^\infty\left ( \partial B_r(0) \right )} \right ).
\label{ModifiedNorm}
\end{equation}

\begin{corollary} \label{Extension2}
For any $h(\theta) = C_\kappa \left ( \sqrt{\kappa - 1}, \left [ \Theta \left (e^{\pm i \theta/2}, 0 \right ) \right ] \right )$, $\lambda > 0$ ,
and $\mu > 0$, there is a positive $\hat{\epsilon} = \hat{\epsilon}(\mu, h) \in (0,1)$ such that
\begin{equation}
\mathcal{Q}' \left ( h, \hat{\epsilon} , \lambda \right ) \subset \mathcal{Q} ' \left ( h, 1, \mu \lambda \right ).
\label{Inclusion2}
\end{equation}
\end{corollary}

\section{Proof of Lemma \ref{KeyLemma}} \label{KeySection}

For $\overline{u} = u \circ \zeta$ and $\overline{h}_\ell = h_\ell \circ \zeta$, $\ell=1,2$, \eqref{Halving} is equivalent to
\begin{equation}
\left \| \overline{u}(r,\theta) - A_0 r^{1/\sqrt{\kappa}} \overline{h}_2(\theta) \right \|_{\lambda, \lambda^3} 
\leq \frac{1}{2} \left \| \overline{u} ( r, \theta ) - A_0 r^{1/\sqrt{\kappa}} \overline{h}_1(\theta) \right \|_{1, \lambda^2},
\label{Halving2}
\end{equation}
when the definition of $\| w \|_{\sigma,\rho}$ is modified as in \eqref{ModifiedNorm}. Hence, proving \eqref{Halving2} suffices.
By choosing an appropriate coordinate system for our target, we can also assume that
$$
\phi(r,\theta) = A_0 r^{1/2\sqrt{\kappa}} h(\theta) = A_0 r^{1/2\sqrt{\kappa}} \left [ \left ( e^{\pm i\theta/2} , 0 \right ) \right ].
$$
Furthermore, we restrict our attention to the case $h(\theta) = \left [ \left ( e^{ i\theta/2} , 0 \right ) \right ]$,
as the case $h(\theta) = \left [ \left ( e^{ - i\theta/2} , 0 \right ) \right ]$ proceeds veribatim.

Suppose there exists a $\lambda \leq \lambda_0$, for which \eqref{Halving2} fails to hold.
Then there exist $\epsilon_i \downarrow 0$ and $h_i \in C^2 \left ( S^1, \mathbf{D}_\kappa \right )$ such  that 
$h_i = \left ( \sqrt{\kappa-1},  n_i(\theta) \right )$, where $n_i(\theta) = \left [ \Theta_i \left ( e^{i \theta/2} , 0 \right ) \right ]$,
and $ A_0 \left \| h_i - h \right \|_{C^2 \left ( S^1 \right )} \leq \epsilon_i$,
and energy-minimizing harmonic maps $\left \{ u_i \right \} \subset \mathcal{Q} \left ( h , \epsilon_i , \lambda \right )$, such that
\begin{equation}
\begin{aligned}
\inf \left \{ \left \| \overline{u}_i(r, \theta) - A_0 r^{1/\sqrt{\kappa}} \overline{\chi}( \theta) \right \|_{\lambda, \lambda^3} \, : \, \chi \in C^2 \left ( S^1, \mathbf{D}_{\kappa} \right ), \; 
\chi = \left [ \hat{\Theta}  \left ( e^{i\theta/2}, 0 \right ) \right ] \right \} \\ 
> \frac{1}{2} \left \| \overline{u}_i(r,\theta) - A_0 r^{1/\sqrt{\kappa}} \overline{h}_i( \theta) \right \|_{1, \lambda^2}.
\label{Negation}
\end{aligned}
\end{equation}

We observe that $\overline{h}$ and $\overline{h}_i$ can be lifted to a map into $\mathbf{C}_\kappa$ given by:
$$ 
\left (\sqrt{\kappa-1}, \overline{n}_i(\theta)  \right ) = \left ( \sqrt{\kappa-1}, \Theta_i \left ( e^{\pm i\theta}, 0 \right ) \right ).
$$
Likewise, since all tangent maps of $\overline{u}_i$ at $0$ can also be lifted to a map into $\mathbf{C}_\kappa$, $\overline{u}_i$ 
themselves can be lifted to maps into $\mathbf{C}_\kappa$.
We remark that the target $\mathbf{C}_\kappa$ is isometrically embedded in $\mathbf{R}^4$ by the inclusion map. 
Therefore, the lifting operations allow us to explicitly carry out extrinsic calculations with ease.
For the remainder of the proof, we relabel the lifted maps $\overline{h}, \overline{n}, \overline{h}_i, \overline{n}_i, \overline{u}_i$  into $\mathbf{C}_\kappa$ 
as $h$, $n$, $h_i$, $n_i$ and $u_i$ respectively, and use the modified norm \eqref{ModifiedNorm} in the corresponding estimates. 

Note that $\left \{ u_i \right \} \subset \mathcal{Q}' \left ( h , \epsilon_i , \lambda \right )$ implies that the left-hand side in \eqref{Negation} is less than $4 \epsilon_i$.
Hence, by invoking an appropriately scaled form of Corollary \ref{Extension2}, we can find a sequence $R_i \downarrow 0$ such that
\begin{equation}
\lim_{i \to 0} \left \| u_i(r,\theta) - A_0 r^{1/\sqrt{\kappa}} h_i(\theta) \right \|_{1, R_i} = 0.
\label{SeqRadii}
\end{equation}

By the Weiss-type estimate \eqref{Weiss}, we have:
\begin{equation}
\int_{B_1(0)} \left | \frac{\partial \left ( u_i/ r^{1/\sqrt{\kappa}} \right )}{\partial r} \right |^2 \, \mathrm{d}x
\leq C \left \| u_i - \phi \right \|_{1, \lambda^2}^2.
\label{Weiss2}
\end{equation}
Arguing as in the proof of Lemma \ref{Extension} with \eqref{Weiss2} and \eqref{SeqRadii} in hand, we can conclude that for any $R \in (0,1)$, there exists a $C(R) > 0$
such that
\begin{equation}
\left \| u_i(r,\theta) - A_0 r^{1/\sqrt{\kappa}} h_i(\theta) \right \|_{1, R} 
\leq C(R) \left \| u_i(r,\theta) - A_0 r^{1/\sqrt{\kappa}} h_i(\theta) \right \|_{1, \lambda^2}.
\label{BEstimate}
\end{equation}

We set
$$\beta_i = \left \| u_i(r,\theta) - A_0 r^{1/\sqrt{\kappa}} h_i(\theta) \right \|_{1, \lambda^2},$$ 
$$w_i(r,\theta) = \beta_i^{-1} \left ( u_i(r,\theta) - A_0 r^{1/\sqrt{\kappa}} h_i(\theta) \right ) . $$  
By \eqref{Weiss2} and \eqref{BEstimate}, $w_i$ satisfies the following estimates:
\begin{equation*}
\begin{aligned}
\int_{B_1(0)} \left | \frac{\partial \left ( w_i/ r^{1/\sqrt{\kappa}} \right )}{\partial r} \right |^2 \, \mathrm{d}x \leq C, \\
\left \| w_i \right \|_{1,R} \leq C(R).
\end{aligned}
\end{equation*}
These estimates enable us to assume, by passing to a subsequence if necessary, that $w_i \to w$ in $C^2_{loc} \left ( B_1(0) \backslash \{ 0 \} \right )$, where
\begin{equation}
\begin{aligned}
\int_{B_1(0)} \left | \frac{\partial \left ( w/ r^{1/\sqrt{\kappa}} \right )}{\partial r} \right |^2 \, \mathrm{d}x \leq C, \\
\left \| w \right \|_{1,R} \leq C(R).
\end{aligned}
\label{GrowthControl}
\end{equation}
Furthermore, using $\beta_i \to 0$ and \eqref{SeqRadii}, it is easy to note that $w \in C^2 \left ( B_1(0) \backslash \{ 0 \} \right )$ satisfies the Jacobi field equation $L_\phi w = 0$
for $\phi(r,\theta) = A_0 r^{1/\sqrt{\kappa}} h(\theta)$ in $B_1(0) \backslash \{ 0 \}$, where $h \, : \, S^1 \to \mathbf{C}_\kappa$ is given by 
$h(\theta) = \left ( \sqrt{\kappa-1}, n(\theta) \right ) = \left ( \sqrt{\kappa-1}, \left ( e^{i \theta} , 0 \right ) \right )$.

As in \eqref{JacobiFullMod}, for $w(r,\theta) = r^{1/\sqrt{\kappa}} \tilde{w}(r,\theta)$, $\tilde{w}$ satisfies
\begin{equation}
\partial_r^2 \tilde{w} + \left ( 1 + \frac{2}{\sqrt{\kappa}} \right ) \frac{1}{r} \partial_r \tilde{w} + \frac{1}{r^2} \mathcal{L}' \tilde{w} = 0,
\label{JacobiExp}
\end{equation}
where $\mathcal{L}'$ is the self-adjoint, elliptic operator defined in \eqref{CircleOperatorMod}.  Hence, freezing $r$ and treating $\tilde{w}(r)(\theta)$ as a map defined on $S^1$,
we can expand it with respect to the orthonormal eigenbasis $\left \{ \tilde{\Psi}_j \right \}$ of $\mathcal{L}'$ 
with $r$-dependent coefficients $a_j$:
$$
\tilde{w}(r) = \sum_{j=1}^\infty a_j(r) \tilde{\Psi}_j.
$$
Substituting this sum in \eqref{JacobiExp}, we obtain the following equation for $a_j$, for each $j \geq 1$,
$$
a_j''(r) + \left ( 1 + \frac{2}{\sqrt{\kappa}} \right ) \frac{1}{r} a_j'(r) - \frac{\lambda_j}{r^2} a_j(r) = 0, \quad 0<r<1.
$$
We observe that for each $j$, $a_j(r) = r^{\gamma_j}$, where $\gamma_j \in \mathbf{C}$, is the solution to this equation for
$
\gamma_j = - \frac{1}{ \sqrt{\kappa}} \pm \sqrt{ \frac{1}{ \kappa} + \lambda_j },
$
when $\lambda \neq - \frac{1}{ \sqrt{\kappa}}$. When $\lambda = - \frac{1}{ \sqrt{\kappa}}$, the solution must be of the form $B + C \log r$. 
Therefore,
\begin{equation}
\begin{aligned}
\tilde{w}(r,\theta) & = \sum_{j \in J_1} A_j r^{\gamma_j} \tilde{\Psi}_j(\theta) + \sum_{j \in J_2} \left ( B_j + C_j \log r \right ) \tilde{\Psi}_j(\theta) \\
& + \sum_{j \in J_3} \left ( D_j \cos \left ( \mathrm{Im} \gamma_j \log r  \right ) + 
E_j \sin \left ( \mathrm{Im} \gamma_j \log r  \right ) \right ) r^{-1/\sqrt{\kappa}} \tilde{\Psi}_j(\theta) ,
\end{aligned}
\label{EigenD}
\end{equation}
for $A_j, B_j, C_j, D_j, E_j \in \mathbf{R}$, and
\begin{equation*}
J_1 = \left \{ j \, : \,  \mathrm{Im}  \gamma_j  = 0, \; \mathrm{Re} \gamma_j \neq 0 \right \},  \quad
J_2 = \left \{ j \, : \, \gamma_j = 0 \right \},  \quad
J_3 = \left \{ j \, : \, \mathrm{Im}  \gamma_j  \neq 0 \right \}.
\end{equation*}
Note that a priori, for each $j$ we have two terms in the sum, corresponding to two conjugate values for $\gamma_j$.
However, observing that \eqref{GrowthControl} is precisely:
\begin{equation*}
\int_{B_1(0)} \left | \frac{\partial  \tilde{w} }{\partial r} \right |^2 \, \mathrm{d}x \leq C,
\end{equation*}
we conclude that $J_3 = \emptyset$, $C_j = 0$ for each $j \in J_2$, $J_1 = \left \{ j \, : \, \gamma_j > 0 \right \}$. 
Thus, we can redefine: $\gamma_j = - \frac{1}{ \sqrt{\kappa}} + \sqrt{ \frac{1}{ \kappa} + \lambda_j } $, $J_1 = \left \{ j \, : \, \lambda_j > 0 \right \}$,
$J_2 = \left \{ j \, : \, \lambda_j = 0 \right \}$, and
$$
\tilde{w}(r, \theta) = \tilde{S}(r,\theta) + \tilde{H}(\theta), \quad 0<r<1,
$$
where
$$
\tilde{S}(r,\theta) = \sum_{j \in J_1} r^{\gamma_j} \tilde{\Psi}_j(\theta), \quad \tilde{H}(\theta) = \sum_{j \in J_2} B_j \tilde{\Psi}_j(\theta).
$$

Firstly, we note that for all $j \in J_1$, $\gamma_j \geq \Gamma = - \frac{1}{\sqrt{\kappa}} + \sqrt{\frac{1}{ \kappa} + \Lambda} > 0$, where $\Lambda$ is the
smallest positive eigenvalue of $\mathcal{L}'$. Therefore, for $S(r,\theta) = r^{1/\sqrt{\kappa}} \tilde{S}(r,\theta)$,
\begin{equation}
\left \| S \right \|_{\lambda, \lambda^3} \leq \lambda^\Gamma \left \| S \right \|_{1, \lambda^2} \leq \lambda_0^\Gamma \left \| S \right \|_{1, \lambda^2},
\label{DecayingPart}
\end{equation}
for $\lambda \leq \lambda_0$, where $\lambda_0$ will be chosen at the last step of our argument. 

Secondly, we note that since for each $j \in J_2$, $\gamma_j = 0$ is equivalent to $\lambda_j =0$, we have:
\begin{equation}
\mathcal{L}' \tilde{H} = 0 \quad \mathrm{on} \; S^1.
\label{JacobiObtained}
\end{equation}
Hence, by our classification of geometrically relevant solutions to \eqref{JacobiObtained} in Section \ref{Linearization},
$\tilde{H}$ is a Jacobi field along the harmonic map $n \, : \, S^1 \to S^2$. 
Moreover, 
\begin{equation}
\begin{aligned}
\| \tilde{H} \|_{L^2 \left ( S^1 \right )}^2 = \sum_{j \in J_2} B_j^2 & = 
\sum_{j \in J_2} \left | \int_{S^1} \tilde{\Psi}_j(\theta) \cdot w(1,\theta) \, \mathrm{d} \theta \right |^2
 \\ & \leq \int_{S^1} |w|^2(1,\theta) \leq 2\pi \|w\|_{1,\lambda^2} \leq 2 \pi,
\label{Hsquared}
\end{aligned}
\end{equation}
by the orthonormality of the eigenbasis $ \left \{ \tilde{\Psi}_j \right \}$, the Cauchy-Schwarz inequality, 
$\left \| w_i \right \|_{1, \lambda^2} = 1$ and that $w_i \to w$ in $C_{\mathrm{loc}}^2 \left ( B_1(0) \backslash \{ 0 \} \right )$.

Consequently, by the discussion in Section \ref{IntegrabilitySection} and Remark \ref{CircleMaps}, there are Jacobi fields $\tilde{H}_i$ along harmonic maps $n_i$ defined above, such that
\begin{equation}
\lim_{i \to \infty} \| \tilde{H}_i - \tilde{H} \|_{C^2 \left ( S^1 \right ) } = 0,
\label{GoodFields}
\end{equation}
and as in \eqref{Taylor}, for $i$ sufficiently large, $ \|\beta_i \tilde{H} \|_{L^2 \left ( S^1 \right )} $ is small enough that
there are harmonic maps $\hat{n}_i \in C^2 \left (S^1, S^1 \right )$ such that
\begin{equation}
\hat{n}_i = n_i + \beta_i A_0^{-1} \tilde{H}_i + o \left ( \beta_i \right ).
\label{Taylor2}
\end{equation}
Note that the constant factor $A_0^{-1}$ in front of $\tilde{H}_i$ is determined by a scaling argument, as $w(r,\theta) = r^{1/\sqrt{\kappa}} \tilde{w} (r,\theta)$ is a Jacobi field along
$\phi(r,\theta) = A_0 r^{1/\sqrt{\kappa}} h(\theta)$, where $h(\theta) = \left ( \sqrt{\kappa - 1}, n(\theta) \right )$.

We let $\hat{h}_i(\theta) = \left ( \sqrt{\kappa - 1}, \hat{n}_i (\theta) \right )$, and using \eqref{GoodFields} and \eqref{Taylor2}, we estimate
\begin{equation}
\begin{aligned}
\left \| A_0 r^{1/\sqrt{\kappa}} \hat{h}_i - A_0 r^{1/\sqrt{\kappa}} h_i - \beta_i r^{1/\sqrt{\kappa}} \tilde{H} \right \|_{\lambda, \lambda^3} \\
 \leq \left \| A_0 r^{1/\sqrt{\kappa}} \hat{h}_i - A_0 r^{1/\sqrt{\kappa}} h_i - \beta_ir^{1/2\sqrt{\kappa}}  \tilde{H}_i \right \|_{\lambda, \lambda^3} +
\beta_i \left \| \tilde{H}_i - \tilde{H} \right \|_{C^2 \left ( S^1 \right )} \\
  \leq A_0 \left \| \hat{n}_i - n_i - \beta_i A_0^{-1} \tilde{H}_i \right \|_{C^2 \left ( S^1 \right )} + \beta_i \left \| \tilde{H}_i - \tilde{H} \right \|_{C^2 \left ( S^1 \right )} \\
  \leq \frac{1}{8} \beta_i,
 \end{aligned}
 \label{FirstOrder}
\end{equation}
for $i$ large enough.

Since $w_i \to w$ in $C^2_{\mathrm{loc}} \left ( B_1(0) \backslash \{ 0 \} \right )$ and $w(r,\theta) = r^{1/\sqrt{\kappa}} \left ( \tilde{S}(r,\theta) + \tilde{H} ( \theta ) \right )$,
for $i$ sufficiently large, we have:
\begin{equation}  
\begin{aligned}
\left \|  u_i - A_0 r^{1/\sqrt{\kappa}} h_i - \beta_i w \right \|_{\lambda, \lambda^3}
 = \beta_i  \left \| w_i - w \right \|_{\lambda, \lambda^3} 
 < \frac{1}{8} \beta_i.
\end{aligned}
\label{IntEstimate}
\end{equation}

Using \eqref{GoodFields}, \eqref{FirstOrder}, \eqref{IntEstimate}, for $i$ large enough, we obtain:
\begin{equation}
\begin{aligned}
\left \| u_i - A_0 r^{1/\sqrt{\kappa}} \hat{h}_i \right \|_{\lambda, \lambda^3} \leq
& \left \|  u_i - A_0 r^{1/\sqrt{\kappa}} h_i - \beta_i w \right \|_{\lambda, \lambda^3} + \\
& \left \| A_0 r^{1/\sqrt{\kappa}} h_i + \beta_i \tilde{H}_i - A_0 r^{1/\sqrt{\kappa}} \hat{h}_i \right \|_{\lambda, \lambda^3} + \\
& \beta_i \left \| \tilde{H} - \tilde{H}_i \right \|_{\lambda, \lambda^3} + \left \| \beta_i S \right \|_{\lambda, \lambda^3} \\
& \leq \frac{1}{8} \beta_i +  \frac{1}{8} \beta_i +  \frac{1}{8} \beta_i + \lambda_0^{\Gamma} \left \| \beta_i S \right \|_{1, \lambda^2},
\end{aligned}
\label{Penultimate}
\end{equation}
where the last inequality is due to \eqref{DecayingPart}.

Finally, we can estimate
\begin{equation}
\| S \|_{1, \lambda^2} \leq \| w \|_{1, \lambda^2} + \| H \|_{1, \lambda^2} \leq 1 + \| \tilde{H} \|_{C^2 \left ( S^1 \right )} \leq 1 + C.
\label{FinalEstimate}
\end{equation}
Here both $\| w \|_{1, \lambda^2}$ and $ \| \tilde{H} \|_{L^2 \left ( S^1 \right )}$ are controlled uniformly as in \eqref{Hsquared}, 
and $\tilde{H}$ solves $\mathcal{L}' \tilde{H} = 0$, the geometrically relevant solutions of which have been classified in Section \ref{Linearization}.
Therefore, the absolute bound on  $ \| \tilde{H} \|_{L^2 \left ( S^1 \right )}$ implies an absolute bound on $ \| \tilde{H} \|_{C^2 \left ( S^1 \right )}$.
(We remark that crucially for the applications of Lemma \ref{KeyLemma}, these bounds remain unchanged, when $h$ differs by a rotation.)

Thus, we obtain:
\begin{equation}
\left \| u_i - A_0 r^{1/\sqrt{\kappa}} \hat{h}_i \right \|_{\lambda, \lambda^3} \leq \frac{3}{8} \beta_i + \lambda_0^{\Gamma} (1 + C) \beta_i.
\label{Ultimate}
\end{equation}
Choosing $\lambda_0 = \left ( 8(1+C) \right )^{-1/ \Gamma}$, we conclude that
$$
\left \| u_i - A_0 r^{1/\sqrt{\kappa}} \hat{h}_i \right \|_{\lambda, \lambda^3} \leq \frac{1}{2} \beta_i,
$$
which contradicts \eqref{Negation}.

\section{Proofs of Lemma \ref{GeometricEstimate} and Theorem \ref{MainTheorem}:} \label{FinalSection}

Firstly, we use Lemma \ref{KeyLemma} to prove Lemma \ref{GeometricEstimate}.

\begin{proof}[{Proof of Lemma \ref{GeometricEstimate}}]

Applying Lemma \ref{KeyLemma} with $h_1 = h$, we obtain a map $h_2$ that satisfies \eqref{Halving}.
We denote $\phi_k(r, \theta) = r^{1/2\sqrt{\kappa}} h_k ( \theta )$, for $k \geq 1$.
We observe that \eqref{Halving} and \eqref{Initial} together give
\begin{equation*}
\begin{aligned}
A_0 \left \| h_1 - h_2 \right \|_{C^2 \left ( S^1 \right )} & \leq \left \| u - A_0 \phi_1 \right \|_{\lambda, \lambda^2} + \left \| u - A_0 \phi_2 \right \|_{\lambda, \lambda^2} \\
& \leq \frac{3}{2} \left \| u - A_0 \phi_1 \right \|_{1, \lambda^2} \leq  \frac{3}{2} \eta \epsilon.
\end{aligned}
\end{equation*}
Therefore, for $\eta < \frac{2}{3}$ we have: $A_0 \left \| h_1 - h_2 \right \|_{C^2 \left ( S^1 \right )} \leq \epsilon$. 
Moreover, defining $u_\lambda = \lambda^{-1/2\sqrt{\kappa}} u \left ( \lambda \cdot \right )$, arguing as in Lemma \ref{Extension} after an appropriate scaling, we get:
\begin{equation}
\left \| u_\lambda - A_0  \phi_2 \right \|_{1, \lambda^3} = \left \| u - A_0  \phi_2 \right \|_{\lambda, \lambda^4} < C ( \lambda )\epsilon,
\label{Hypo1}
\end{equation}
where
$
\lim_{\lambda \to 0} C ( \lambda ) = + \infty.
$
Thus, for $\eta < 2/3$, we have:
$$
A_0 \left \| h_1 - h_2 \right \|_{C^2 \left ( S^1 \right )} \leq \epsilon,
$$  
and
$$
u_\lambda \in \mathcal{Q} \left ( h_2, C(\lambda) \epsilon, \lambda \right ).
$$
In particular, we can apply Lemma \ref{KeyLemma} to $u_\lambda$ and $h_2$,
with the slightly modified hypothesis that $u_\lambda \in \mathcal{Q} \left ( h_2, C(\lambda) \epsilon, \lambda \right )$.
We note that the only difference this modification leads to is that in \eqref{GrowthControl}, the second bound $C(R)$ deteriorates to $C(R,\lambda)$,
where
$
\lim_{\lambda \to 0} C (R, \lambda ) = + \infty.
$
However, the finiteness of $C(R,\lambda)$ for each $\lambda > 0$ is sufficient for the rest of the proof of Lemma \ref{KeyLemma}.
Therefore, its conclusion remains unchanged. 
Consequently, we obtain maps $h_3$ and corresponding $\phi_3$ such that
\begin{equation*}
\begin{aligned}
\left \| u - A_0 \phi_3 \right \|_{\lambda^2, \lambda^4} 
= \left \| u_\lambda - A_0 \phi_3 \right \|_{\lambda, \lambda^3} 
< \frac{1}{2} \left \| u_\lambda - A_0 \phi_2 \right \|_{1, \lambda^2} 
& = \frac{1}{2} \left \| u - A_0 \phi_2 \right \|_{\lambda, \lambda^3} \\
& < \frac{1}{4} \left \| u - A_0 \phi_1 \right \|_{1, \lambda^2}.
\end{aligned}
\end{equation*}

Hence, applying Lemma \ref{KeyLemma} to $u_{\lambda^k}$ repeatedly, 
(while modifying its hypothesis as: $u_{\lambda^k} \in \mathcal{Q} \left ( h_k, C(\lambda) \epsilon, \lambda^k \right )$, 
and replacing $h$ with $h_k$ in each step), 
for $k \geq 1$, by induction, we obtain  a sequence of maps $h_k$, $\phi_k$, $k \geq 1$ such that
\begin{equation*}
\left \| h_k - h_{k+1} \right \|_{C^2 \left ( S^1 \right )} \leq  A_0^{-1} \frac{3}{2^k} \eta \epsilon,
\end{equation*}
and
\begin{equation*}
\left \| u - A_0 \phi_{k+1} \right \|_{\lambda^k, \lambda^{k+2}} < \frac{1}{2^k}  \eta \epsilon.
\end{equation*}
Consequently, there exists an $\hat{h} \in C^2 \left (S^1; \mathbf{D}_\kappa \right )$ such that $h_k \to \hat{h}$ in $C^2 \left ( S^1 \right )$ with the convergence rate,
\begin{equation*}
\left \| \hat{h} - h_k \right \|_{C^2 \left ( S^1 \right )} = \lim_{j \to \infty} \left \| h_k - h_j \right \|_{C^2 \left ( S^1 \right )} 
\leq \sum_{\ell=k}^j \left \| h_\ell - h_{\ell+1} \right \|_{C^2 \left ( S^1 \right )}  < 2 A_0^{-1} \frac{3}{2^k} \eta \epsilon. 
\end{equation*}
As a result, we have:
\begin{equation*}
\begin{aligned}
\left \| u - A_0 \hat{\phi} \right \|_{\lambda^k, \lambda^{k+2}} 
\leq \left \| u - A_0 \phi_{k+1} \right \|_{\lambda^k, \lambda^{k+2}} + A_0 \left \| h_{k+1} - \hat{h} \right \|_{C^2 \left ( S^1 \right )}
& \leq \frac{1}{2^k} \eta \epsilon + \frac{3}{2^k} \eta \epsilon \\
& = \frac{16}{2^{k+2}} \eta \epsilon.
\end{aligned}
\end{equation*}
Choosing $\mu = \frac{ \log 2}{\log ( 1 / \lambda ) }$, by the definition of the norm $\| \cdot \|_{a,b}$, we get:
\begin{equation*}
\left \| u(r,\cdot) - A_0 r^{1/2\sqrt{\kappa}} \hat{h} \right \|_{C^2 \left ( S^1 \right )} \leq C \eta \epsilon r^{\frac{1}{2\sqrt{\kappa}}+\mu}.
\end{equation*}
\end{proof}

Now using Lemma \ref{GeometricEstimate}, we finally prove Theorem \ref{MainTheorem}.

\begin{proof}[{Proof of Theorem \ref{MainTheorem}}]

Since $\phi \in C^\infty \left ( \mathbf{R}^2, \mathbf{D}_\kappa \right )$ is a homogeneous blow-up of $u$ at $0$, for any $\epsilon > 0$, $\eta > 0$, $\lambda > 0$,
there exists a sequence $\left \{ \rho_i \right \} \downarrow 0$ and an $i \left ( \epsilon , \eta \right ) > 0$ such that
for every $i \geq i \left ( \epsilon, \eta,  \lambda \right )$,
$$ 
\left | \left \| u \left ( \rho_i \cdot \right )  \right \|_{L^2 \left ( B_1 (0) \right ) }^{-1} u \left ( \rho_i r, \theta \right ) - \phi(r,\theta ) \right | 
< \lambda^{1/\sqrt{\kappa}} \eta \epsilon /6, \quad 0 < r < 1,
$$
which implies
$$
\left | \left ( \rho_i r \right )^{-1/2\sqrt{\kappa}} \left ( u \left ( \rho_i r, \theta \right ) - A \left (0, \rho_i \right ) \phi (\rho_i r, \theta) \right ) \right | < 
A \left (0, \rho_i \right ) \left ( \lambda^2 /r \right )^{1/2\sqrt{\kappa}} \eta \epsilon / 6.
$$
Consequently,
$$
\left | |x|^{-1/2\sqrt{\kappa}} \left ( u(x) - A \left (0, \rho_i \right ) \phi(x) \right ) \right | < A \left (0, \rho_i \right ) \eta \epsilon/3, \quad \rho_i \lambda^2 < |x| < \rho_i. 
$$

Choosing $\rho = \rho_i$, where $i = i \left ( \epsilon, \eta, \lambda, \kappa \right )$ large enough and defining: 
$
u_\rho (x)= \rho^{-1/2\sqrt{\kappa}} u \left ( \rho x \right ),
$
we obtain:
$$
\left | |x|^{-1/2\sqrt{\kappa}} \left ( u_\rho (x) - A ( 0, \rho ) \phi(x) \right ) \right | < A(0,\rho) \eta \epsilon/3, \quad \lambda^2 < |x| < 1.
$$
Note that at this stage $\rho$ depends on $\lambda$, $\epsilon$, $\eta$, and $\kappa$, while $\lambda$, $\eta$ and $\epsilon$ are arbitrary. 

We control the higher-order terms in the definition of $\left \| u_\rho - A (0, \rho) \phi \right \|_{1, \lambda^2}$ by similarly estimating
$\left \| \tau^{-1/2\sqrt{\kappa}} \left [ u_\rho (\tau \cdot ) - A(0,\rho) \phi (\tau \cdot ) \right ] \right \|_{L^2 \left ( S^1 \right )} $ for every $\tau \in \left ( \lambda^2, 1 \right )$.
These estimates can be obtained by the argument in the proof of Lemma \ref{Extension} and the regularity theory of \cite{SU}, 
for a suitably chosen $\lambda \left ( \epsilon, \eta \right ) \in (0,1)$.
Hence, we get:
\begin{equation*}
\left \| u_\rho - A(0,\rho) \phi \right \|_{1, \lambda^2} < A(0,\rho) \eta \epsilon,
\end{equation*}
where $\rho = \rho \left ( \epsilon, \eta, \lambda, \kappa \right ) \in (0,1)$.
Similarly, we can obtain:
\begin{equation*}
\left \| u_\rho - A(0,\rho) \phi \right \|_{1, \lambda^3} < A(0,\rho) \epsilon.
\end{equation*}
by suitably modifying $\lambda \left ( \epsilon, \eta \right ) \in (0,1)$ first, and then shrinking $\rho = \rho \left ( \epsilon, \eta, \lambda, \kappa \right ) \in (0,1)$ further, if necessary.

Note that we can apply Lemma \ref{GeometricEstimate} with $\epsilon$ replaced with $A(0,\rho) \epsilon$, 
since $A(0,\rho)$ is bounded by Remark \ref{Monotonicity} and is equal to $A_0$ in \eqref{AChoice} for $u_\rho$.
Therefore, by Lemma \ref{GeometricEstimate} and scaling, we have for $\tilde{r} \in (0,1)$,
\begin{equation*}
\left \| u (\rho \tilde{r},\theta) - A (0,\rho) \phi( \rho \tilde{r},\theta) \right \|_{C^2 \left ( S^1 \right )} 
\leq C A(0,\rho) \eta \epsilon  (\rho \tilde{r})^{\frac{1}{2\sqrt{\kappa}} + \mu }.
\end{equation*}
Thus, setting $r_0 = \rho$ and $r = \rho \tilde{r}$, we conclude that for every $r \in \left ( 0, r_0 \right )$,
\begin{equation*}
\left | u (r,\theta) - A \left ( 0, r_0 \right ) \phi( r,\theta) \right | \leq C A \left ( 0, r_0 \right ) r^{\frac{1}{2\sqrt{\kappa}} + \mu }.
\end{equation*}
\end{proof}

\section*{Acknowledgment}
I would like to thank my thesis advisor Professor Fang-Hua Lin for suggesting this problem and many helpful discussions.

\bibliography{bibext}
\bibliographystyle{amsplain}

\end{document}